\documentclass[psamsfonts,12pt, draft]{amsart}
\usepackage{amssymb,amsmath}
\usepackage{amsthm}
\usepackage{mathrsfs}
\usepackage{enumerate,indentfirst}
\usepackage[fleqn,tbtags]{mathtools}
\usepackage{color}
\usepackage{courier}
\usepackage{hyperref}
\usepackage[a4paper,left=2.5cm,right=2.5cm,top=3.2cm,bottom=3.2cm]{geometry}
\usepackage{multicol}
\usepackage{xr}
\usepackage{multicol}
 \newtheorem{theorem}{Theorem}[section]
 \newtheorem{corollary}[theorem]{Corollary}
 \newtheorem{lemma}[theorem]{Lemma}
 \newtheorem{proposition}[theorem]{Proposition}

 \theoremstyle{definition}
 \newtheorem{example}[theorem]{Example}
 
 \theoremstyle{remark}
  \numberwithin{equation}{section}

\renewcommand{\phi}{\varphi}
\renewcommand{\theta}{\vartheta}

\DeclareMathOperator{\tform}{\mathfrak{t}}

\DeclareMathOperator{\wform}{\mathfrak{w}}

\DeclareMathOperator{\trace}{Tr}

\DeclarePairedDelimiterX\sipt[2]{(}{)_{\tform}}{#1\,\delimsize\vert\,#2}
\DeclarePairedDelimiterX\sipv[2]{(}{)_{v}}{#1\,\delimsize\vert\,#2}
\DeclarePairedDelimiterX\sipw[2]{(}{)_{w}}{#1\,\delimsize\vert\,#2}

\newcommand{\alg}{\mathscr{A}}

\newcommand{\abs}[1]{\lvert#1\rvert}
\newcommand{\dupN}{\mathbb{N}}

\newcommand{\seq}[1]{(#1_{n})_{n\in\dupN}}

\newcommand{\nen}{n\in\mathbb{N}}

\newcommand{\dupC}{\mathbb{C}}

\newcommand{\dom}{\operatorname{dom}}

\newcommand{\ran}{\operatorname{ran}}

\newcommand{\hil}{H}

\newcommand{\hilf}{H_f}
\newcommand{\hilg}{H_g}

\newcommand{\bh}{\mathrm{B}(\hil)}
\newcommand{\kh}{\mathrm{B}_0(\hil)}
\newcommand{\bhf}{\mathrm{B}(\hilf)}

\newcommand{\rpa}{\alg^{\sharp}}
\newcommand{\acsa}{a^{*}a}

\newcommand{\amb}{(a-b)^{*}(a-b)}
\newcommand{\bcsb}{b^{*}b}

\DeclarePairedDelimiterX\sip[2]{(}{)}{#1\,\delimsize\vert\,#2}
\DeclarePairedDelimiterX\siptilde[2]{(}{)_{\!_{\widetilde{A}}}}{#1\,\delimsize\vert\,#2}
\DeclarePairedDelimiterX\sipf[2]{(}{)_{f}}{#1\,\delimsize\vert\,#2}
\DeclarePairedDelimiterX\sipg[2]{(}{)_{g}}{#1\,\delimsize\vert\,#2}
\DeclarePairedDelimiterX\siptw[2]{(}{)_{\tform+\wform}}{#1\,\delimsize\vert\,#2}
\DeclarePairedDelimiterX\set[2]{\{}{\}}{#1\,\delimsize\vert\,#2}
\DeclarePairedDelimiterX\dual[2]{\langle}{\rangle}{#1,#2}
\DeclarePairedDelimiterX\sipa[2]{(}{)_{\!_A}}{#1\,\delimsize\vert\,#2}
\DeclarePairedDelimiterX\sipc[2]{(}{)_{\!_C}}{#1\,\delimsize\vert\,#2}
\DeclarePairedDelimiterX\sipab[2]{(}{)_{\!_{A+B}}}{#1\,\delimsize\vert\,#2}
\DeclarePairedDelimiterX\sipb[2]{(}{)_{\!_B}}{#1\,\delimsize\vert\,#2}

\newcommand{\limn}{\lim\limits_{n\rightarrow\infty}}

\newcommand{\hsh}{\mathrm{B}_2(\hil)}
\newcommand{\trh}{\mathrm{B}_1(\hil)}

\allowdisplaybreaks
\begin{document}

\title{On the order structure of representable functionals}

\author[Zs. Tarcsay]{Zsigmond Tarcsay}

\address{%
Zs. Tarcsay, Department of Applied Analysis\\ E\"otv\"os L. University\\ P\'azm\'any P\'eter s\'et\'any 1/c.\\ Budapest H-1117\\ Hungary}

\email{tarcsay@cs.elte.hu}

\author[T. Titkos]{Tam\'as Titkos}
\address{T. Titkos, Alfr\'ed R\'enyi Institute of Mathematics, Hungarian Academy of Sciences\\ Re\'altanoda utca 13-15.\\ Budapest H-1053\\ Hungary }
\email{titkos.tamas@renyi.mta.hu}
\thanks{Tam\'as Titkos was supported by the National Research, Development and Innovation Office, NKFIH-104206
}


\begin{abstract} The main purpose of this paper is to investigate some natural problems regarding the order structure of representable functionals on $^*$-algebras. We describe the extreme points of order intervals, and give a nontrivial sufficient condition to decide whether or not the infimum of two representable functionals exists. To this aim we offer a suitable approach to the Lebesgue decomposition theory, which is in complete analogy with the one developed by Ando in the context of positive operators. This tight analogy allows to invoke  Ando's results to characterize uniqueness of the decomposition, and solve the infimum problem over certain operator algebras.
\end{abstract}

\maketitle

\section{Introduction} 
Let $\alg$ be a not necessarily unital ${}^{\ast}$-algebra. A linear functional $f:\alg\to\mathbb{C}$ is said to be representable if there is a $^*$-representation $\pi$ of $\alg$ in a Hilbert space $\hil$ and a vector $\zeta\in\hil$ such that
\begin{align*}
f(a)=\sip{\pi(a)\zeta}{\zeta},\qquad a\in\alg.
\end{align*}
The set $\rpa$ of representable functionals on $\alg$ is partially ordered with respect to the ordering induced by positivity. The purpose of this paper is to investigate some structural questions of $\rpa$, namely to determine the extreme points of order intervals of the form $[f,g]$, and to decide whether or not the greatest lower bound of two representable functionals exists.

Analogous questions for positive operators were investigated because of their physical importance. Namely, in the standard Hilbert space model for quantum mechanics (in the so-called effect algebra), the greatest lower bound $A\wedge B$ is the greatest effect so that the probability that $A\wedge B$ is observed is not greater than the probability that both $A$ and $B$ are observed for every state of the system. Sharp quantum effects, i.e. the extreme points of the effect algebra  represent the perfectly accurate yes-no measurements in quantum probability theory. For more details we refer the reader to \cite{M-G,Pekarev}, and also to \cite{EL} for the corresponding potential theoretic aspects.

The question that whether or not the infimum $A\wedge B$ of two positive operators $A$ and $B$  exists has been a long standing open problem. The finite dimensional case was solved by Moreland and Gudder in \cite{M-G}. In the general case, a necessary and sufficient condition was given by Ando in \cite{Ando-inf} by means of the Lebesgue-type decomposition \cite{Ando}. This approach turned out to be effective also by characterizing the extreme points of the effect algebra, or more generally, of operator intervals. Since Lebesgue decomposition has a quite extensive literature also in the context of representable functionals (see \cite{Gudder,Kosaki,Szucs,Tarcsay_repr}), it seems possible to treat these questions from a similar point of view. To do so, we provide a self-contained exposition of the Lebesgue decomposition theory on general $^*$-algebras. In fact, given two representable functionals $f$ and $g$, we give a Lebesgue-type decomposition of $f$ by means of $[g]f := \sup\limits_{n\in\dupN} f : (ng)$, where $f : (ng)$ is the parallel sum of $f$ and $ng$.

The idea of obtaining the absolutely continuous part by means of the parallel sum is taken from \cite{Hassi2009} of Hassi, Sebesty\'en, and de Snoo, who developed a general Lebesgue decomposition theory for nonnegative sesquilinear forms. Actually, it is always a possibility to consider forms instead of functionals (set $\mathfrak{t}_f(a,b):=f(b^*a)$), to invoke some results concerning forms, and then to return to functionals. But we have to keep in mind that not every nonnegative form is induced by a representable functional, and hence, the last step may fail or at least may be a source of difficulties. On the other hand, the set of forms and the set of functionals may have completely different order structures. (As we shall see, this is indeed the situation in some cases, depending on the properties of $\alg$.) In order to evade all these complications and also to be self-contained we  present an exposition which avoids any reference to forms.

Let us describe the content of this paper in detail. Section \ref{S: preliminaries} is devoted to collect all the notions and notations that are needed to formulate our results. Section \ref{S: Lebesgue decomposition} includes a concise presentation of the Lebesgue decomposition theorem in the context of representable functionals. Our interpretation is of algebraic nature and follows the arguments of \cite{Hassi2009}. We then turn to describe the structure of order intervals in Section \ref{S: extremal section}. In fact, extreme points of intervals are described by means of Lebesgue decomposition, and by the notion of \lq\lq disjoint part".
Section \ref{ S: infimum} is intended to investigate the infimum  of representable functionals. In analogy with Ando's result concerning the existence of the infimum of positive operators, a necessary and sufficient condition was given in \cite{titkosinf} for sesquilinear forms. Since $^*$-algebras may have very different algebraic structures, there is no hope to give a similar satisfactory characterization in this setting. Nevertheless, we give a sufficient condition for the existence of the infimum of two representable functionals. This condition allows us to show that the infimum of two extreme points of an order interval always exists. This condition is however not necessary in general as it is demonstrated by some counterexamples. In the last section, we  discuss the uniqueness of the Lebesgue-type decomposition. According to Kosaki \cite{Kosaki}, the Lebesgue decomposition of representable functionals is not necessarily unique, even in the case of von Neumann algebras. In this aspect, we are going to provide a sufficient condition under which the aforementioned decomposition is unique and we  demonstrate examples where the condition may also be necessary. Some other examples show that there is no uniqueness over Hilbert algebras, as well as for the von Neumann algebra $\bh$, where $\hil$ an infinite dimensional Hilbert space.


\section{Preliminaries}\label{S: preliminaries}

Throughout this paper we fix a (not necessarily unital) $^*$-algebra $\alg$. A linear functional $f:\alg\to\mathbb{C}$ is said to be \emph{positive} if $f(a^{*}a)\geq0$ for all $a\in\alg$. Let $f$ and $g$ be positive functionals, we write $f\leq g$ if $g-f$ is positive. 

A $^*$-representation of $\alg$ in a Hilbert space $\hil$ is an algebra homomorphism from $\alg$ to $\bh$, the $C^*$-algebra of all bounded operators of $\hil$, which also preserves the involution: $\pi(a^*)=\pi(a)^*$. A linear functional $f$ on $\alg$ is said to be representable if there is a $^*$-representation $\pi$ of $\alg$ in a Hilbert space $\hil$ and a vector $\zeta\in\hil$ such that
\begin{align}\label{E: representing}
f(a)=\sip{\pi(a)\zeta}{\zeta},\qquad a\in\alg.
\end{align}
Note that a representable functional is automatically positive, i.e., $f(a^*a)\geq0$ for all $a\in\alg$; the converse is not always true in general. Recall that a positive functional $f$ on $\alg$ is representable if and only if it satisfies the following two conditions (see eg. \cite{Sebestyen84}): on the one hand, there exists a constant $M\geq0$ such that
\begin{gather}\label{E:cyclic}
    \abs{f(a)}^2\leq Mf(a^*a)\qquad \textrm{for all $a\in\alg$},
\end{gather}
and, on the other hand, for any $b\in\alg$ there exists $M_b\geq0$ such that
\begin{gather}\label{E:repr}
    f(a^*b^*ba)\leq M_{b}f(a^*a)\qquad \textrm{for all $a\in\alg$}.
\end{gather}
In that case the GNS construction provides a representing triple $(\hilf, \pi_f,\zeta_f)$ fulfilling \eqref{E: representing}. Note also that $\zeta_f$ here is the Riesz representing vector of the bounded linear functional $\pi_f(a)\zeta_f\mapsto f(a)$, whence we get the equality
\begin{gather}\label{E:hilbertbound}
    \|\zeta_f\|_f^2=\sup\big\{|f(a)|^2\,\big|\,f(\acsa)\leq 1,~a\in\alg\big\}.
\end{gather}
In other words, $\|f\|_H:=\|\zeta_f\|_f^2$ is the smallest constant satisfying 
\begin{align*}
|f(a)|^2\leq \|f\|_Hf(a^*a) \qquad \textrm{for all $a\in\alg$},
\end{align*}
and is called the Hilbert bound of $f$. 
Recall also the important fact that representability is hereditary in the sense that if $h$ is any Hilbert bounded positive functional such that $h\leq g$ for a $g\in\rpa$ then $h$ and $g-h$ are both representable. Furthermore, $\|h\|_H\leq\|g\|_H$ holds also true in that case; for the proof see \cite[Proposition 9.4.22 (e)]{palmer}. As an easy consequence of this fact one concludes that any order bounded monotonically increasing sequence $\seq{f}$ in $\rpa$ possesses a supremum in $\rpa$. That is, if $f_n\leq f_{n+1}\leq g$ is satisfied for $\seq{f}$ and $g$ in $\rpa$ then the pointwise limit $f$ equals the least upper bound of $\seq{f}$ in  $\rpa$.

Let us recall the notions of absolute continuity and singularity of representable functionals: we say that $f$ is \emph{$g$-absolutely continuous} (we write $f\ll g$) if \begin{equation*}
g(a_n^*a_n)\to0 \quad\textrm{and} \quad f((a_n-a_m)^*(a_n-a_m))\to0\quad \textrm{imply} \quad f(a_n^*a_n)\to0
\end{equation*}
for any sequence $\seq{a}$ of $\alg$. Observe immediately that if $f$ is \emph{$g$-dominated}, i.e., $f\leq cg$ with some $c>0$ then $f$ is $g$-absolutely continuous. We say that $f$ and $g$ are \emph{singular} ($f\perp g$) if $h\leq f$ and $h\leq g$ imply $h=0$ for any representable functional $h$.

The following characterization of absolute continuity will be used frequently in the sequel. 
\begin{theorem}\label{T:closabs}
For $f,g\in\rpa$ the following assertions are equivalent:
\begin{enumerate}[\upshape (i)]
 \item $f$ is $g$-absolutely continuous;
 \item $f$ equals the pointwise limit of a monotonically increasing sequence $\seq{f}$ in $\rpa$ satisfying $f_n\leq \alpha_ng$ for each $n$ with some $\alpha_n\geq 0$. 
\end{enumerate}
\end{theorem}
\begin{proof}
Assume first that $f$ is $g$-absolutely continuous. This means that the map $J:\hil_g\to\hil_f$,
defined by the correspondence 
\begin{align}\label{E:J*pi}
\pi_g(a)\zeta_g\mapsto\pi_f(a)\zeta_f,\qquad a\in\alg,
\end{align}
is a well-defined closable operator. First we claim that  
\begin{align}\label{E:komm}
\pi_f(a)\langle\dom J^*\rangle\subseteq \dom J^*\quad \textrm{and}\quad J^*\pi_f(a)z=\pi_g(a)J^*z
\end{align}
hold for all $a\in\alg$ and $z\in\dom J^*$. Indeed, for $z\in\dom J^*$ and $a,x\in \alg$ we have 
\begin{align*}
\sipf{J\pi_g(x)\zeta_g}{\pi_f(a)z}&=\sipf{\pi_f(x)\zeta_f}{\pi_f(a)z}=\sipf{\pi_f(a^*x)\zeta_f}{z}\\
&=\sipf{J\pi_g(a^*x)\zeta_g}{z}=\sipg{\pi_g(a^*)\pi_g(x)\zeta_g}{J^*z}\\
&=\sipg{\pi_g(x)\zeta_g}{\pi_g(a)J^*z}.
\end{align*}
Consider the (typically unbounded) positive selfadjoint operator $S:=J^{*}J^{**}$ on $\hil_g$ and denote by $E$ its spectral measure. By \eqref{E:komm} we infer that $S$ commutes with every $\pi_g(a)$ in the usual sense that  $\pi_g(a)S\subset S\pi_g(a)$ for all $a\in\alg$.  Each $\pi_g(a)$ being bounded we can also deduce that the spectral projections of $S$  commute with the $\pi_g(a)$'s. Consequently, setting $\displaystyle S_n:=\int_0^n t\,dE(t)$ for each integer $n$ we get $\pi_g(a)S_n=S_n\pi_g(a)$ and hence 
\begin{align*}
f_n(a):=\sipg{S_n\pi_g(a)\zeta_g}{\zeta_g},\qquad a\in\alg,
\end{align*}
is a representable functional such that $f_n\leq f_{n+1}$, 
\begin{align*}
f_n(a^*a)&=\sipg{S_n^{1/2}\pi_g(a)\zeta_g}{S_n^{1/2}\pi_g(a)\zeta_g}\to \sipg{S^{1/2}\pi_g(a)\zeta_g}{S^{1/2}\pi_g(a)\zeta_g}=f(a^*a).
\end{align*}
Finally,
\begin{align*}
f_n(a^*a)=\sipg{S_n\pi_g(a)\zeta_g}{\pi_g(a)\zeta_g}\leq \|S_n\|\sipg{\pi_g(a)\zeta_g}{\pi_g(a)\zeta_g}=\|S_n\|g(a^*a).
\end{align*}
Thus $\seq{f}$ fulfills each condition of (ii). For the converse, suppose (ii). First of all observe that $J:\hil_g\to\hilf$ defined by \eqref{E:J*pi} is well defined: if $\pi_g(a)\zeta_g=0$ for some $a\in\alg$ then  $f_n(a^*a)=0$ by assumption. Hence 
\begin{align*}
\|\pi_f(a)\zeta_f\|^2_f=f(a^*a)=\limn f_n(a^*a)=0,
\end{align*}
indeed. Our claim is to show that $J$ is closable, or equivalently, that $J^*$ is densely defined. By the Riesz representation theorem, for any integer $n$ there exists a unique positive contraction $C_n\in\bhf$  such that 
\begin{equation*}
\sipf{C_n(\pi_f(a)\zeta_f)}{\pi_f(b)\zeta_f}=f_n(b^*a),\qquad a,b\in\alg.
\end{equation*}
We claim that $\ran C_n\subseteq \dom J^* $ (here $\ran$ refers to the range). Let $\xi\in\hilf$ be fixed and observe for any $a\in\alg$ that
\begin{align*}
\abs{\sipf{J\pi_g(a)\zeta_g}{C_n\xi}}^2&=\abs{\sipf{\pi_f(a)\zeta_f}{C_n\xi}}^2=\abs{\sipf{C_n(\pi_f(a)\zeta_f)}{\xi}}^2\\
 &\leq \|\xi\|^2_f\|C_n(\pi_f(a)\zeta_f)\|^2_f\\ 
 &\leq  \|\xi\|^2_f\|C_n\|\sipf{C_n(\pi_f(a)\zeta_f)}{\pi_f(a)\zeta_f}\\
&=\|\xi\|^2_f\|C_n\|f_n(a^*a)\\
&\leq \alpha_n\|\xi\|^2_f g(a^*a)= \alpha_n\|\xi\|^2\sipg{\pi_g(a)\zeta_g}{\pi_g(a)\zeta_g},
\end{align*}
whence $C_n\xi\in\dom J^*$. To conclude the statement our only duty is to verify the identity
\begin{equation}\label{E:ker}
\bigcap_{n=1}^{\infty}\ker C_n =\{0\}.
\end{equation}
To this end observe that $C_n(\pi_f(a)\zeta_f)\to\pi_f(a)\zeta_f$ as $n\to\infty$ for each $a$ in $\alg$:
\begin{align*}
\|\pi_f(a)\zeta_f-&C_n(\pi_f(a)\zeta_f)\|^2_f\\ 
&=f(a^*a)-2\sipf{C_n(\pi_f(a)\zeta_f)}{\pi_f(a)\zeta_f}+\|C_n(\pi_f(a)\zeta_f)\|^2_f\\
&\leq 2f(a^*a)-2f_n(a^*a)\to0.
\end{align*}
Consequently, $(C_n)_{n\in\dupN}$ converges strongly to the identity operator on $\hilf$. This obviously gives \eqref{E:ker}. Hence $\dom J^*$ is dense, i.e., $J$ is closable.
\end{proof}

Condition (ii) above was called absolute continuity in \cite{Ando} in the context of positive operators, and almost dominatedness in \cite{Hassi2009} for nonnegative sesquilinear forms. Condition (i) was referred to as strong absolute continuity in \cite{Gudder} and as closability in \cite{Hassi2009}. We shall follow Ando's terminology and use the phrase absolute continuity for both (i) and (ii).

The most important notion in this paper is the so called \emph{parallel sum} of two representable functionals, which was introduced in \cite{Tarcsay_parallel}.
Let $f$ and $g$ be arbitrary elements of $\rpa$ and denote by $(\hilf,\pi_f,\zeta_f)$ and $(\hil_g,\pi_g,\zeta_g)$ the GNS-triples associated with $f$ and $g$, respectively. Let $\pi$ stand for the direct sum of $\pi_f$ and $\pi_g$. Then the closure of the following subspace
\begin{equation*}
\set{\pi_f(a)\zeta_f\oplus\pi_g(a)\zeta_g}{a\in\alg}\subseteq\hilf\oplus\hil_g
\end{equation*}
     is $\pi$-invariant. If $P$ stands for the orthogonal projection onto its ortho-comp\-lement, then the functional $f:g$ defined  by the correspondence
\begin{equation*}
        (f:g)(a):=\sip{\pi(a)P(\zeta_f\oplus0)}{P(\zeta_f\oplus0)},\qquad a\in\alg,
\end{equation*}
is representable and its values on positive elements are given by
\begin{equation}\label{E:f:g_q}
        (f:g)(a^*a)=\inf\set{f((a-b)^*(a-b))+g(b^*b)}{b\in\alg},\qquad a\in\alg.
\end{equation}
For the details see \cite[Theorem 5.1]{Tarcsay_parallel}.

The following proposition contains three elementary observations regarding the parallel sum that will be used constantly in upcoming sections.
\begin{proposition}\label{T: parsprop=}
Let $f,g,f_1,f_2,g_1,g_2\in\rpa$ be representable functionals on $\alg$. Then
\begin{enumerate}[\upshape (a)]
\item $f:g\leq f$ and $f:g\leq g$,
\item $f_1\leq f_2$ and $g_1\leq g_2$ imply $f_1:g_1\leq f_2:g_2$,
\item if $g\in\rpa$ and $\alpha,\beta>0$, then $(\alpha g):(\beta g)=\frac{\alpha\beta}{\alpha+\beta}g$.
\end{enumerate}
\begin{proof}
Properties $(a)$ and $(b)$ follow easily from (\ref{E:f:g_q}), we prove only $(c)$. An elementary computation gives that
$\alpha g(\amb)+\beta g(\bcsb)$ equals to
\begin{equation*}
\frac{\alpha\beta}{\alpha+\beta} g(\acsa)+g\left(\Big(\sqrt{\alpha+\beta}b+\frac{\alpha}{ \sqrt{\alpha+\beta} }a\Big)^{*}\Big(\sqrt{\alpha+\beta}b+\frac{\alpha}{\sqrt{\alpha+\beta}}a\Big)\right).\end{equation*}
Since every term is nonnegative, the infimum is attained at $b:=\frac{-\alpha}{\alpha+\beta}a$.\end{proof}
\end{proposition}
One of the advantages of parallel addition is that it characterizes the notion of singularity. Namely, $f$ and $g$ are singular precisely when $f:g=0$. Indeed, if $f$ and $g$ are singular, then $f:g=0$ because of Proposition \ref{T: parsprop=} (a). Conversely, if $f:g=0$ and $h$ is a representable functional majored by both $f$ and $g$, then $h=2(h:h)\leq2(f:g)=0$ holds according to Proposition \ref{T: parsprop=} (b) and (c). This characterization can be found also in \cite[Theorem 5.2]{Tarcsay_parallel}.
Now, we can define the representable functional $[g]f$ by 
\begin{equation}\label{E:[g]f_def}
        [g]f:=\sup\limits_{\nen}f:(ng),
     \end{equation}
in  complete analogy with the case of bounded positive operators \cite{Ando}. Observe that $[g]f$ is $g$-absolutely continuous according to Theorem \ref{T:closabs}. Furthermore, it is clear from \eqref{E:f:g_q} and the definition of $[g]f$  that $[g]f\leq f$, $f_1\leq f_2$ and $g_1\leq g_2$ imply $[g_1]f_1\leq [g_2]f_2$, and that $[g](cf)=c[g]f$ for all $c\geq0$. 

\section{Lebesgue decomposition}\label{S: Lebesgue decomposition}
Let $f$ and $g$ be representable functionals on $\alg$. A decomposition $f=f_1+f_2$ is called a Lebesgue-type decomposition of $f$ with respect to $g$ (or shortly a \emph{$g$-Lebesgue decomposition of $f$}) if $f_1,f_2\in\rpa$, $f_1\ll g$, and $f_2\perp g$. In order to investigate some extremal problems, we present a suitable approach for this type of decomposition. This approach is in complete analogy to the one developed by Ando \cite{Ando} for positive operators, and the one given by Hassi, Sebesty\'en, and de Snoo for nonnegative sesquilinear forms \cite{Hassi2009}. 

The significance of the functional $[g]f$ in Lebesgue decomposition theory becomes clear already in the next proposition.

\begin{proposition}\label{P: reg-sing-characterization}
Let $f$ and $g$ be representable functionals on $\alg$. Then
\begin{enumerate}[\upshape (a)]
\item $f\perp g$ if and only if $[g]f=0$,
\item $f\ll g$ if and only if $[g]f=f$.
\end{enumerate}
\end{proposition}

\begin{proof}
To prove $(a)$ observe that $[g]f=0$ implies $f:g=0$ by definition, and hence, $f$ and $g$ are singular. Conversely, assume that $f$ and $g$ are singular, but $[g]f\neq0$. In this case, there exists $\nen$ such that $f:ng\neq0$, which is a contradiction, because $0\neq\big(\frac{1}{n}f\big):g\leq f,g$. 

Now we are going to prove $(b)$. If $[g]f=f$ then $f:(ng)\leq ng$ and Theorem \ref{T:closabs} imply that $f$ is $g$-absolutely continuous. For the converse implication observe first that if $f$ is $g$-dominated, i.e., there exists an $\alpha > 0$ constant such that $f\leq \alpha g$, then $[g]f=f$. Indeed, for every $\nen$ we have
$$f\geq [g]f=\sup\limits_{k\in\dupN}f:(kg)\geq f:\left(\frac{n}{\alpha}f\right)=\frac{n}{\alpha+n}f,$$
which implies $[g]f=f$ by taking supremum in $n$. Now assume that $f$ is $g$-absolutely continuous, and recall that this guarantees that $f$ is a limit of a monotone increasing sequence $(f_n)_{\nen}$ of $g$-dominated representable functionals. According to the previous observation, we have
$$f_n=[g]f_n\leq[g]f\leq f$$
which implies that $[g]f=f$, again by taking supremum.
\end{proof}
\begin{corollary}
If $f$ is both $g$-absolutely continuous and $g$-singular, then $f$ is the identically zero functional.
\end{corollary}

In the next theorem (which is the main result of this section) we establish a Lebesgue-type decomposition in terms of $[g]f$; cf. also \cite{Gudder}, \cite{Szucs}, \cite{Tarcsay_repr}.

\begin{theorem}\label{T: Ldt-rf}
Let $f,g\in\rpa$ be arbitrary representable functionals on $\alg$. Then 
\begin{equation}\label{E:lebdec}
f=[g]f+(f-[g]f)
\end{equation}
is a $g$-Lebesgue decomposition of $f$, that is, $[g]f\ll g$ and $(f-[g]f)\perp g$. Furthermore, this decomposition is extremal in the following sense:
$$h\in\rpa,~h\leq f~\mbox{and}~h\ll g\quad\Rightarrow\quad h\leq[g]f.$$
\end{theorem}
\begin{proof}
First we prove the maximality of $[g]f$. Let $h$ be a representable functional such that $h\leq f$ and $h\ll g$. According to Proposition \ref{P: reg-sing-characterization}(b) and Proposition \ref{T: parsprop=}(b) we have
$$h=[g]h=\sup\limits_{\nen}h:(ng)\leq \sup\limits_{\nen}f:(ng)=[g]f.$$
Using the maximality property of $[g]f$ and the fact that the sum of $g$-absolutely continuous functionals is $g$-absolutely continuous, one can obtain $$[g](h+k)\geq [g]h+[g]k$$ for every $h,k\in\rpa$. Since $[g]f$ is $g$-absolutely continuous according to Theorem \ref{T:closabs}, it is enough to prove the $g$-singularity of $f-[g]f$. This follows from Proposition \ref{P: reg-sing-characterization} and
\begin{align*}
[g]f=[g]\geq[g]([g]f)+[g](f-[g](f))=[g]f+[g](f-[g]f),
\end{align*}
because $[g](f-[g]f)=0$ holds if and only if $g$ and $f-[g]f$ are singular.
\end{proof}

The equivalence in Theorem \ref{T:closabs} combined with the maximality of $[g]f$ shows that decomposition \eqref{E:lebdec} coincides with the one presented in \cite[Theorem 3.3]{Tarcsay_repr}. Using unbounded operator-techniques, it was also proved in \cite[Theorem 4.2]{Tarcsay_repr} that the corresponding absolutely continuous parts $[g]f$ and $[f]g$ are mutually absolute continuous for all representable functionals $f$ and $g$. Proposition \ref{P: reg-sing-characterization} and Theorem \ref{T:closabs} shed a new light to this result: mutual absolute continuity follows directly from the order properties of $\rpa$.

\begin{corollary}\label{abszfolytabszfolyt}
Let $f$ and $g$ be representable functionals on $\alg$, and consider their Le\-besgue-type decompositions with respect to each other. Then the absolutely continuous parts are mutually absolutely continuous, i.e.,
\begin{align*}
[g]f\ll[f]g\qquad\mbox{and}\qquad[f]g\ll[g]f.
\end{align*}
\end{corollary}
\begin{proof}
Observe first that if $g_1$, $g_2$, and $h$ are representable functionals such that $g_1\leq g_2$ and $g_1\ll h$, then 
\begin{align*}g_1\ll[g_2]h.\end{align*}
Indeed, if $g_1$ is $h$-absolutely continuous, then there exists a monotonically nondecreasing sequence of representable functionals $(g_{1,n})_{n\in\mathbb{N}}$ such that $\sup\limits_{n\in\mathbb{N}}g_{1,n}=g_1$ and $g_{1,n}$ is dominated by $h$ for all $n\in\mathbb{N}$ (i.e., $g_{1,n}\leq g_1\leq g_2$, and $g_{1,n}\leq c_n h$ for some $c_n\geq 0$). Consequently,
\begin{align*}
g_{1,n}=[g_2]g_{1,n}\leq[g_2]c_nh=c_n[g_2]h
\end{align*}
which means that $g_1\ll[g_2]h$. Now, apply the previous observation with $$g_1:=[g]f\qquad g_2:=f,\qquad h:=g.$$
Similarly, one can prove that $[f]g\ll[g]f$.
\end{proof}

Combining the previous result with Proposition \ref{P: reg-sing-characterization}(b) we gain the following identity
\begin{align}\label{F: [g]f=[[f]g]f}
[g]f=\big[[f]g\big]f.
\end{align}
Indeed, on the one hand, $g\geq[f]g$ implies  $[g]f\geq\big[[f]g\big]f$. On the other hand, the inequalities $[g]f\ll[f]g$ and $[g]f\leq f$ yield $$[g]f=\big[[f]g\big]([g]f)\leq\big[[f]g\big]f\leq[g]f.$$
\section{Extreme points of functional intervals}\label{S: extremal section}
In this chapter we are going to describe the extreme points of intervals of the form 
$$[f,g]:=\set[\big]{h}{h\in\rpa,~f\leq h\leq g},$$
where $f,g\in\rpa$, $f\leq g$. To do so we need first the following lemma characterizing singularity. 

\begin{lemma}\label{L: reg+sing}
Let $f,g,h\in\rpa$ be representable functionals and assume that $f$ is $g$-dominated. Then the following statements are equivalent
\begin{enumerate}[\upshape (i)]
\item $g\perp h$,
\item $[g](f+h)=[g]f=f$.
\end{enumerate}
\end{lemma}
\begin{proof}
Since $g\perp h$ is equivalent with $[g]h=0$, the implication (ii)$\Rightarrow$(i) follows from
$$[g]f=[g](f+h)\geq[g]f+[g]h\geq[g]f.$$
To prove the converse implication we observe first that $$f=[g]f\leq[g](f+h)\leq f+h,$$ hence $[g](f+h)=f+k$ for some $k\in\rpa$, $k\leq h$. Assume indirectly that $k\neq0$. Since $g\perp h$ and $k\leq h$, we have $g\perp k$. Consequently, $k$ may not be $g$-absolutely continuous, because the only representable functional which is simultaneously $g$-singular and $g$-absolutely continuous is $k=0$. Thus there exists a sequence $(a_n)_{\nen}$ in $\alg$ such that $k((a_n-a_m)^*(a_n-a_m))\to0$ and $g(a_n^*a_n)\to0$, but $k(a_n^*a_n)\nrightarrow0$. Since $f$ is $g$-dominated, $f(a_n^*a_n)\to0$ holds for this sequence, and thus $f((a_n-a_m)^*(a_n-a_m))\to0$. But then we get  
$$(f+k)((a_n-a_m)^*(a_n-a_m))\to0,\qquad g(a_n^*a_n)\to0\qquad(f+k)(a_n^*a_n)\nrightarrow0,$$
which is impossible because $f+k$ is $g$-absolutely continuous.
\end{proof}

Assume that $f$ and $g$ are representable functionals such that $g\leq f$. Following the terminology of F. Riesz \cite{Riesz} we say that $g$ is a \emph{disjoint part} of $f$ if $g$ and $f-g$ are mutually singular. The following theorem states that the extreme points of $[0,f]$ are precisely the disjoint parts of $f$. For analogous results see \cite{anderson-trapp,EL,green-morley,Pekarev,sakai}.
\begin{theorem}\label{T: exchar}
Let $f$ and $g$ be representable functionals, and assume that $g\leq f$. Then the following statements are equivalent.
\end{theorem}
\begin{enumerate}[\upshape (i)]
\item $g$ is an extreme point of $[0,f]$,
\item $g$ is a disjoint part of $f$,
\item $[g]f=g$.
\end{enumerate}
\begin{proof}
To prove that (i) implies (ii) assume that $g$ is not singular with respect to $f-g$. In this case,  $g:(f-g)\neq0$, hence 
$$\frac{1}{2}\Big(g-\big(g:(f-g)\big)\Big)+\frac{1}{2}\Big(g+\big(g:(f-g)\big)\Big)=g$$
is a nontrivial convex combination in $[0,f]$. Implication (ii)$\Rightarrow$(iii) follows directly from Proposition \ref{P: reg-sing-characterization}(b) and Lemma \ref{L: reg+sing}:
$$[g]f=[g](g+(f-g))=[g]g=g.$$
Finally, assume  $[g]f=g$ and, indirectly, that  $g$ is not an extreme point of $[0,f]$. In this case, there exists $h,k\in[0,f]$ ($h\neq k$) and $0<\lambda<1$ such that $g=\lambda h+(1-\lambda)k$. According to Proposition \ref{T: parsprop=}, we have
$$f:(ng)\geq h:(n\lambda h)=\frac{n\lambda}{n\lambda+1}h,$$
which yields $g=[g]f\geq h$. Similarly,
$$f:(ng)\geq k:(n(1-\lambda)k) =\frac{n(1-\lambda)}{n(1-\lambda)+1}k, $$
which yields $g=[g]f\geq k.$
These imply $g=h=k$, otherwise there exists an $a\in\alg$ such that $g(\acsa)>h(\acsa)$ or $g(\acsa)>k(\acsa)$. But  then,
$$g(\acsa)=\lambda g(\acsa)+(1-\lambda)g(\acsa)>\lambda h(\acsa)+(1-\lambda)k(\acsa)=g(\acsa),$$
which is contradiction.
\end{proof}
As an elementary consequence, we can characterize the extreme points of $[f,g]$ in terms of the parallel sum as follows:
\begin{corollary}
A representable functional $h\in[f,g]$ is an extreme point of $[f,g]$ if and only if
$$(h-f):(g-h)=0. $$  
\end{corollary}
\begin{corollary}
For $f,g\in\rpa$ the following assertions are equivalent:
\begin{enumerate}[\upshape (i)]
 \item $g$ is an extreme point of $[0,g+f]$,
 \item Each extreme point of $[g,g+f]$ is an extreme point of $[0,g+f]$.
\end{enumerate}
\end{corollary}
\begin{proof} To prove (i)$\Rightarrow$(ii) assume that $g$ is an extreme point of $[0,f+g]$. According to Theorem \ref{T: exchar}(ii) this means that $g$ and $f$ are singular. Let $k$ be an arbitrary extreme point of $[g,g+f]$, and observe that $k$ can be written as $k=g+h$, where $h$ is an extreme point of $[0,f]$. Since $h\leq f$ and $f\perp g$, Lemma \ref{L: reg+sing}(ii) and \eqref{F: [g]f=[[f]g]f} ensure that
\begin{align*}
[g+h](f-h)&=\big[[f-h](g+h)\big](f-h)\\ 
          &\leq\big[[f](g+h)\big](f-h)=[h](f-h)=0.
\end{align*}
Again, using Theorem \ref{T: exchar}(ii) we conclude that $k$ is an extreme point of $[0,f+g]$ as well. Implication (ii)$\Rightarrow$(i) is trivial.
\end{proof}
\section{On the infimum of representable functionals}\label{ S: infimum}

The significance of non-commutative Lebesgue decomposition theory is revealed by considering the infimum of positive operators in a Hilbert space. Namely, Ando in his famous paper \cite{Ando-inf} has shown that the infimum $A\wedge B$ of the positive operators $A$ and $B$ exists (in the positive cone) precisely when the corresponding absolutely continuous parts $[A]B$ and $[B]A$ are comparable. Because of the strong analogy, a natural question arises: can we handle the problem for representable functionals in the same way? Unfortunately, it turns out that the situation here is much more complicated. Since ${}^{\ast}$-algebras may have very different algebraic structures, there is no hope to give a similar satisfactory characterization in such a general setting. Nevertheless a nontrivial sufficient condition can be furnished in the same manner. As we shall see, this condition may be necessary, as well as redundant in some important particular cases (see examples below).

First of all, for the sake of completeness, let us recall the notion of infimum in the partially ordered set $\rpa$. We say that the infimum of $f,g\in\rpa$ exists in $\rpa$ if there is an element $h$ in $\rpa$ such that $h\leq f$, $h\leq g$, and that $h'\leq f$ and $h'\leq g$ imply $h'\leq h$ for any $h'\in\rpa$. The infimum of $f$ and $g$ (in case when it exists) is denoted by $f\wedge g$. 

Translating Ando's result we gain the following sufficient condition for the existence of $f\wedge g$.

\begin{theorem}\label{T: inf_sufficient}
Let $f,g$ be representable functionals over $\alg$. If $[f]g$ and $[g]f$ are comparable, that is $[f]g\leq[g]f$ or $[g]f\leq[f]g$, then the infimum $f\wedge g$ exists in $\rpa$. In this case, $f\wedge g=\min\{[f]g,[g]f\}.$
\end{theorem}
\begin{proof}
Assume that $[f]g\leq [g]f$. In this case $[f]g\leq[g]f\leq f$ and $[f]g\leq g$ hold, thus $[f]g\in[0,f]\cap[0,g]$. Now, let $h\in\rpa$ be a functional such that $h\leq f$ and $h\leq g$. This implies immediately that $h\ll f$ and $h\leq g$, and hence $h\leq[f]g$ by the maximality property of $[f]g$. And similarly, if $[g]f\leq[f]g$, then $f\wedge g=[g]f$.
\end{proof}
As was proved in \cite{M-G} the set of effects  is not a lattice, but the infimum of a sharp effect with any other effect exists. The previous sufficient condition combined with Theorem \ref{T: exchar} allows us to prove an analogous result for representable functionals.
\begin{corollary}
Let $u$ be an extreme point of the interval $[0,f]$. Then the infimum $u\wedge h$ exists for all $h\in[0,f]$, and 
\begin{align*} 
u\wedge h=[u]h.
\end{align*}
If $u$ and $h$ are both extreme points, then 
\begin{align*}u\wedge h=[u]h=[h]u.
\end{align*}
\end{corollary}
\begin{proof}
According to Theorem \ref{T: inf_sufficient} it is enough to show that $[u]h\leq [h]u$. On the one hand, we obtain from $h\leq f$ and Theorem \ref{T: exchar} (iii) that $[u]h\leq [u]f=u$. Since $[u]h$ is $h$-absolutely continuous, we have
\begin{align*}
[u]h=[h]([u]h)\leq[h]u,
\end{align*}
which means that $u\wedge h=[u]h$. On the other hand, if both $u$ and $h$ are extreme point, we conclude by symmetry that
\begin{align*}
[u]h=u\wedge h=h\wedge u =[h]u.
\end{align*}
\end{proof}

Remark that $[f]g\leq[g]f$ holds precisely when $[f]g\leq f$, thus the sufficient condition in Theorem \ref{T: inf_sufficient} can be stated also as follows: if either $[f]g\leq f$ or $[g]f\leq g$  then the infimum $f\wedge g$ exists.

As indicated above, comparability of $[f]g$ and $[g]f$ is in general not a necessary condition of the existence of the infimum: 
\begin{example}\label{Ex:commutative}
Let $\alg$ be a unital commutative $C^*$-algebra. Then, by the commutative Gelfand-Naimark theorem and the Riesz representation theorem, each positive (hence representable) functional $f$  on $\alg$ is of the form 
\begin{align*}
f(a)=\int \widehat{a}~d\mu_f,\qquad a\in\alg,
\end{align*}
where $\widehat{a}$ is the image of $a$ under the Gelfand transform, $\mu_f$ is a positive, regular Borel measure and the integral is taken over the maximal ideal space of $\alg$. It is well known that nonnegative finite measures on a given $\sigma$-algebra form a lattice. Recall that the infimum of two measures $\mu$ and $\nu$ can be calculated by the formula 
\begin{align*}
(\mu\wedge\nu)(E):=\inf_{F}\big(\mu(E\cap F)+\nu(E\setminus F)\big).
\end{align*}
The Riesz representation theorem also states that the mapping $f\mapsto \mu_f$ is surjective and bipositive, i.e., $f\geq0$ if and only if $\mu_f\geq0$. In particular, the infimum of two positive functionals $f,g\in\rpa$ always exists as the inverse image of $\mu_f\wedge\mu_g$ along this map.
\end{example}
Next we are going to analyze a highly noncommutative example, the (non-unital) $C^*$-algebra of compact operators. This is a good example to demonstrate that the condition given in Theorem \ref{T: inf_sufficient} can be also necessary.
\begin{example}\label{Ex:compact}
Let $\hil$ be a complex Hilbert space and denote by $\kh$ the $C^*$-algebra of all compact operators on $\hil$. Then the topological dual of $\kh$ can be identified with $\trh$, the set of all trace class operators on $\hil$. In particular, any positive (hence representable) functional $f$ on $\kh$ is represented by a positive trace class operator $F\in\trh$ in the sense that 
\begin{align*}
f(T)=\trace(FT), \qquad T\in\kh,
\end{align*}
see \cite{KR1}. Let denote this representation by $\Phi$, i.e.,
\begin{align}\label{E: Phi}
\Phi:\kh^{\sharp}\to\trh;\qquad\Phi(f):=F.
\end{align}
First observe that $\Phi$ is surjective and bipositive ($f\geq0$ if and only if $F\geq0$), thus the infimum of $f,g\in\kh^{\sharp}$ exists in $\kh^{\sharp}$ if and only if the infimum of $F,G$ exists in $\trh$. Since $F,G\in\trh$ one has that $$[0,F]\cap[0,G]=[0,F]\cap[0,G]\cap\trh,$$  the infimum $F\wedge G$ exists if and only if $[F]G$ and $[G]F$ are comparable, according to Ando \cite[Theorem 4]{Ando-inf}. Since the image of $[f]g$ under the representing map is just $[F]G$ (see  Theorem \ref{T:patak} below), we deduce that $f\wedge g$ exists if and only if $[f]g$ and $[g]f$ are comparable.  
\end{example}
Since the missing step in the preceding example will be used also in the next section, we formulate it separately.
\begin{theorem}\label{T:patak}
Let $\hil$ be a Hilbert space and let denote by $\kh$ the $C^{\ast}$-algebra of compact operators on $\hil$. For the representing map $\Phi:\kh^{\sharp}\to\trh$
\begin{align}
f(T)=\trace(\Phi(f)T),\qquad T\in\kh
\end{align}
we have the following assertions.
\begin{enumerate}[\upshape a)]
 \item $f,g\in\kh^{\sharp}$ are mutually singular if and only if $\Phi(f),\Phi(g)$ are so;
 \item $f$ is $g$-absolutely continuous if and only if $\Phi(f)$ is $\Phi(g)$-absolutely continuous. 
\end{enumerate}
In particular, $\Phi([g]f)=[\Phi(g)]\Phi(f)$.
\end{theorem}
\begin{proof}
Assertion a) follows directly from the fact that $\Phi$ is surjective and bipositive. In order to see b), for $x,y\in\hil$ consider the rank one operator $x\otimes y$, defined by $(x\otimes y)(z)=\sip{z}{y}x, z\in\hil$. An immediate calculation shows that 
\begin{align*}
\trace(x\otimes y)=\sip{x}{y},\qquad \textrm{and} \qquad A\circ(x\otimes y)=(Ax)\otimes y,
\end{align*}
for any $A\in\bh$. Hence we obtain 
\begin{equation}\label{E:Fxx}
\sip{\Phi(f)x}{x}=f(x\otimes x)
\end{equation}
for any $f\in\kh^{\sharp}$ and $x\in \hil$. Assume now that $f$ is $g$-absolutely continuous, and consider a monotonically increasing sequence $\seq{f}$ in $\kh^{\sharp}$ such that $f_n\leq \alpha_ng$ for some $\alpha_n\geq 0$ and that $f_n\to f$ pointwise on $\kh$. Then, by bipositivity, $(\Phi(f_n))_{\nen}$ is also  monotonically increasing, $\Phi(f_n)\leq \alpha_n\Phi(g)$, and by \eqref{E:Fxx}, 
\begin{align*}
\sip{\Phi(f_n)x}{x}\to \sip{\Phi(f)x}{x}, 
\end{align*}
for all $x\in\hil$. This implies that $\Phi(f)$ is $\Phi(g)$-absolutely continuous.  Assume conversely that $\Phi(f)$ is $\Phi(g)$-absolutely continuous. Again by bipositivity, there exists a monotonically increasing sequence $\seq{f}$ in $\rpa$ such that $f_n\leq \alpha_ng$ for some $\alpha_n\geq 0$ and that $\sip{\Phi(f_n)x}{x}\to \sip{\Phi(f)x}{x}$ holds for any $x\in\hil$.  To conclude that $f$ is a $g$-absolutely continuous functional, it is enough to prove that $f_n(T)\to f(T)$ for each compact operator $T$. Moreover we prove that 
$f_n\to f$ in functional norm. With this end, let $\varepsilon>0$ and choose an orthonormal basis $E$ in $\hil$. Fix a finite subset $E_0$ of $E$ and an integer $N$ such that 
\begin{align*}
\trace(\Phi(f))-\sum_{e\in E_0}\sip{\Phi(f)e}{e}<\frac{\varepsilon}{2}\quad\mbox{and}\quad
\sum_{e\in E_0}\sip{(\Phi(f)-\Phi(f_n))e}{e}<\frac{\varepsilon}{2}.
\end{align*}
Then for any integer $n$ with $n\geq N$
\begin{align*}
\trace(\Phi(f-f_n))&\leq \trace(\Phi(f-f_N))\\ &\leq \sum_{e\in E\setminus E_0} \sip{\Phi(f)e}{e}+\sum_{e\in E_0}\sip{(\Phi(f)-\Phi(f_n))e}{e}<\frac{\varepsilon}{2}+\frac{\varepsilon}{2}=\varepsilon.
\end{align*}
Hence, we have $\|f_n-f\|=\trace(\Phi(f-f_n))\to0$. It remains only to show that $\Phi([g]f)=[\Phi(g)]\Phi(f)$. It follows immediately from part b), the bipositivity of $\Phi$, and the maximality of the almost dominated part (both in the operator and functional case). Indeed, $\Phi([g]f)$ is $\Phi(g)$-absolutely continuous because $[g]f$ is $g$-absolutely continuous. Hence $\Phi([g]f)\leq[\Phi(g)]\Phi(f)$ by maximality.
To prove the converse implication let us denote by $\Psi$ the inverse of $\Phi$, and observe that $\Psi([\Phi(g)]\Phi(f))$ is $g$-absolutely continuous. And hence, again by maximality, $$\Psi([\Phi(g)]\Phi(f))\leq[g]f.$$ Applying $\Phi$ to both sides gives $[\Phi(g)]\Phi(f)\leq\Phi([g]f)$, completing the proof.
\end{proof}
To illustrate how sophisticated the order structure of $\rpa$ may be, let us consider the following example.
\begin{example}
Let $X$ be a compact Hausdorff space and $\hil$ a Hilbert space. Denote by $C(X)$ the set of all continuous functions from $X$ to $\dupC$, and by $\kh$ the set of all compact operators on $\hil$. Set $\alg$ for their direct sum, that is, $\alg:=C(X)\oplus \kh$, which becomes a $C^*$-algebra with respect to the usual operations, involution, and norm. Then $\alg'=C(X)'\oplus \kh'$. It is easy to check from Example \ref{Ex:commutative} that the infimum  $(f\oplus 0)\wedge (g\oplus0)$ for $f,g\in C(X)^{\sharp}$ always exists. Clearly, $(f\oplus 0)\wedge (g\oplus0)=(f\wedge g)\oplus 0$ in this case. On the other hand, $(0\oplus f)\wedge(0\oplus g)$ for $f,g\in \kh^{\sharp}$ may exist if and only if $f\wedge g$ in $\kh^{\sharp}$ exists. In view of Example \ref{Ex:compact}, this happens only if $[f]g$ and $[g]f$ are comparable. 
\end{example}
\section{Uniqueness of the Lebesgue decomposition}\label{S: uniqueness}

As it has been pointed out by Kosaki \cite{Kosaki}, the Lebesgue-type decomposition of representable functionals is not necessarily unique even over von Neumann algebras.
Nevertheless, it is possible to give a nontrivial sufficient condition for the uniqueness in terms of the regular part. As we shall see, this property can be necessary in some cases.
\begin{theorem}\label{T:sufficient_leb}
Let $f$ and $g$ be representable functionals on $\alg$. If $[g]f\leq cg$ holds for some $c\geq0$ then the Lebesgue-type decomposition of $f$ with respect to $g$ is unique.
\end{theorem}
\begin{proof}
Assume that $f=f_1+f_2$ is a $g$-Lebesgue decomposition of $f$. By the maximality of $[g]f$ we obtain that $[g]f-f_1\in\rpa$, and the inequalities
\begin{align*}
0\leq [g]f-f_1\leq [g]f \leq cg\qquad\mbox{\textrm{and}} \qquad [g]f-f_1\leq f-f_1=f_2
\end{align*}
imply that $[g]f-f_1$ is both absolutely continuous and singular. This yields $[g]f=f_1$.
\end{proof}
 We remark that a similar argument shows that if absolute continuity is a hereditary property (i.e., $f\ll g$ and $h\leq f$ imply $h\ll g$), then the Lebesgue decomposition is unique.
Based on the preceding theorem, the uniqueness problem can be easily solved over finite dimensional $^*$-algebras.
\begin{example}\label{Ex:finitedim}
Let $\alg$ be a finite dimensional $^*$-algebra. Then for any $f\in\rpa$ the representation Hilbert space $\hilf$, obtained along the GNS construction, is finite dimensional. Hence, for $f,g\in\rpa$, the mapping $$\pi_g(a)\zeta_g\mapsto\pi_f(a)\zeta_f$$ from $\hilg$ to $\hilf$ is closable if and only if it is continuous. This means that $f\ll g$ if and only if $f\leq cg$ for some $c\geq 0$. Hence, in view of Theorem \ref{T:sufficient_leb}, the Lebesgue-type decomposition on $\rpa$ is unique.  
\end{example}
In the next example we shall see that uniqueness may also appear in infinite dimension, namely over $C(X)$. In contrast to the finite dimensional case, the sufficient condition of Theorem \ref{T:sufficient_leb} is redundant. 

\begin{example}
Let $\alg$ be a unital commutative $C^{\ast}$ algebra. Since there is a bipositive isometric isomorphism  between representable functionals and nonnegative regular Borel measures (see Example \ref{Ex:commutative}),
absolute continuity is a hereditary property in this case. Indeed, assume that $f$ is $g$-absolutely continuous, and $h\leq f$. According to Theorem \ref{T:closabs}, the representing map $f\mapsto \mu_f$ gives a monotonically decreasing sequence of measures such that $\mu_{f_n}\uparrow\mu_f$, and $\mu_{f_n}$ is $\mu_g$-dominated for all $n\in\mathbb{N}$. Consequently, the inverse image of the sequence $\mu_{h_n}:=\mu_h\wedge\mu_{f_n}$ guarantees that $h$ is $g$-absolutely continuous.
\end{example}
The next example shows that the sufficient condition of Theorem \ref{T:sufficient_leb} may be necessary  in certain infinite dimensional cases.

\begin{example}
For $f,g\in\kh^{\sharp}$ we state that the $g$-Lebesgue decomposition of $f$ is unique precisely when $[g]f\leq cg$ for some $c\geq 0$. Indeed, letting $F=\Phi(f)$ and $G=\Phi(g)$ as in \eqref{E: Phi}, Theorem \ref{T:patak} says that the $g$-Lebesgue decomposition of $f$ is unique if and only if the $G$-Lebesgue decomposition of $F$ is unique in $\bh$. On account of \cite[Theorem 6]{Ando},
the latter assertion is equivalent to $[G]F\leq cG$, and hence, by bipositivity of $\Phi$ to $[g]f\leq cg$.
\end{example}

As the preceding example demonstrates, the Lebesgue decomposition of representable functionals is not necessarily unique over (noncommutative infinite dimensional) $C^{*}$-algebras. Next we are going to prove a similar characterization over $\hsh$, which shows that the Lebesgue decomposition is not necessarily unique even over Hilbert algebras. The main tools are  Ando's uniqueness result, again, and a counterpart of Theorem \ref{T:patak}.
\begin{example}
Let $\hil$ be a Hilbert space and denote by $\hsh$ the Hilbert algebra of Hilbert--Schmidt operators, endowed with inner product
\begin{eqnarray}\label{E:siphsh}
\sip{S}{T}_2=\trace(T^*S),\qquad S,T\in\hsh.
\end{eqnarray}
According to the Riesz representation theorem, any bounded positive functional $f$ on $\hsh$ is of the form
\begin{eqnarray}\label{E:f_T}
f(T)=\sip{T}{F}_2=\trace(FT),\qquad T\in\hsh,
\end{eqnarray}
for some uniquely determined positive linear operator $F\in\hsh$. The mapping $f\mapsto F$ is bipositive, i.e., $f_1\leq f_2$ if and only if $F_1\leq F_2$. Note that the positive functional $f$ is representable if and only if $F$ is a trace class operator, or equivalently, $F^{1/2}\in\hsh$. 
Indeed, if  $F\in\trh$ then
\begin{align*}
\abs{f(T)}^2&= \abs{\sip{TF^{1/2}}{F^{1/2}}_2}^2\leq\|TF^{1/2}\|_2^2\|F^{1/2}\|_2^2=\trace(F) f(T^*T)
\end{align*}
for all $T\in\hsh$, hence $f$ is representable. For the converse suppose $f$ is representable, that is,
\begin{align*}
\abs{f(T)}^2\leq Mf(T^*T),\qquad T\in\hsh
\end{align*}
holds for some  $M\geq0$.  Choose an orthonormal basis $E$ of  $\hil$. Then for any finite set $E_0$ of $E$, denoting by $P$  the orthogonal projection of $\hil$ onto the subspace spanned by $E_0$ we get 
\begin{align*}
\sum_{e\in E_0}\sip{Fe}{e}=\trace(FP)=f(P)=f(P^*P).
\end{align*}
Hence, by representability, 
\begin{align*}
\left(\sum_{e\in E_0}\sip{Fe}{e}\right)^2=\abs{f(P)}^2\leq Mf(P^*P)=M \sum_{e\in E_0}\sip{Fe}{e}.
\end{align*}
This implies 
\begin{eqnarray*}
\sum_{e\in E_0}\sip{Fe}{e}\leq M,
\end{eqnarray*}
whence we get $F\in\trh$ and $\trace(F)\leq M$.

We see therefore that $\hsh^{\sharp}$ may be identified with positive elements of $\trh$ along the bipositive map $f\mapsto F$. Imitating the proof of Theorem \ref{T:patak} it can be shown that this map preserves absolute continuity and singularity as well. Hence, for $f,g\in\hsh^{\sharp}$ the image of $[g]f$ is just $[G]F$. Applying Ando's uniqueness result \cite[Theorem 6]{Ando}, the Lebesgue decomposition of $f$ with respect to $g$ is unique if and only if $[g]f\leq cg$ for some $c\geq0$. 
\end{example}
Our final example demonstrates that Lebesgue decomposition is non-unique over the von Neumann algebra $\bh$. In accordance with Theorem \ref{T:closabs} and Ando's result \cite[Theorem 6]{Ando} we  show that non-uniqueness occurs whenever $\hil$ is infinite dimensional.

\begin{example}
Let $\hil$ be a Hilbert space. Given two positive trace class operators $F,G\in\trh$ we set
\begin{equation}\label{E:normal}
f(T):=\trace(FT), \qquad g(T):=\trace(GT),\qquad T\in\bh.
\end{equation}
Then clearly, $f,g$ are positive (hence representable) functionals on $\bh$. Note that not every positive functional on $\bh$ derives from a trace class operator. Functionals of the form \eqref{E:normal} are called normal functionals. The reader can easily verify that $f\leq g$ if and only if $F\leq G$. Furthermore, any positive functional majorized by a normal one is normal itself: if $k\in\bh^{\sharp}$, $k\leq f$, then $k(T)=\trace(KT)$ for some (unique) positive operator $K\in\trh$. According to these observations and repeating the proof of Theorem \ref{T:patak} one can easily deduce that $F$ and $G$ are mutually singular if and only if $f$ and $g$ are so. Similarly, $F$ is $G$-absolutely continuous if and only if $f$ is $g$-absolutely continuous. In particular, the $g$-Lebesgue decomposition of $f$ is unique precisely if the $G$-Lebesgue decomposition of $F$ is unique.

After this brief comment let us construct a counterexample. Assume $\hil$ is infinite dimensional and consider an orthonormal sequence $\seq{e}$ in it. Let $\seq{\alpha}, \seq{\beta}$ be two monotone decreasing sequences in $\ell^1$ with positive coefficients such that $\alpha_n/\beta_n\to\infty$. Set
\begin{equation}
Fx:=\sum_{n=1}^{\infty} \alpha_n\sip{x}{e_n}e_n,\qquad Gx:=\sum_{n=1}^{\infty} \beta_n\sip{x}{e_n}e_n,\qquad x\in\hil,
\end{equation}
and define $f,g\in\bh^{\sharp}$ by \eqref{E:normal}. Letting 
\begin{align*}
F_nx:=\sum_{k=1}^n \alpha_k\sip{x}{e_k}e_k
\end{align*}
we have $F_n\leq F_{n+1}\to F$ in operator norm and also $F_n\leq \dfrac{\alpha_n}{\beta_n} G$. Hence $F$ is $G$-absolutely continuous, i.e., $[G]F=F$. On the other hand, $F\leq cG$ is impossible because  $\alpha_n/\beta_n\to\infty$. The $G$-Lebesgue decomposition of $F$ is therefore not unique, in accordance with \cite[Theorem 6]{Ando}. Hence the $g$-Lebesgue decomposition of $f$ fails to be unique as well. 
\end{example}

\end{document}